\documentclass[pdflatex,sn-mathphys-num]{sn-jnl}

\newcommand{\aq}{\mathbf{a}}
\newcommand{\wq}{\mathbf{w}}
\newcommand{\vq}{\mathbf{v}}
\newcommand{\eq}{\mathbf{e}}

\usepackage[T1]{fontenc}
\usepackage[utf8]{inputenc}
\usepackage{bbm}
\usepackage{amsmath, amsthm, amssymb, amsfonts}
\usepackage{hyperref}
\usepackage{enumitem}
\usepackage{mathrsfs}
\usepackage{tikz}
\usepackage{graphicx}

\newtheorem{theorem}{Theorem}[section]
\newtheorem{lemma}[theorem]{Lemma}
\newtheorem{proposition}[theorem]{Proposition}
\newtheorem{corollary}[theorem]{Corollary}
\theoremstyle{definition}
\newtheorem{definition}[theorem]{Definition}

\theoremstyle{remark}

\setlist[itemize]{
  topsep=6pt,
  partopsep=0pt,
  itemsep=2pt,
  parsep=0pt,
  labelindent=\parindent,
  leftmargin=*,
  align=left
}

\setlist[enumerate,1]{
  topsep=6pt,
  partopsep=0pt,
  itemsep=2pt,
  parsep=0pt,
  labelindent=\parindent,
  leftmargin=*,
  align=right
}

\begin{document}

\title[Dynamics, Random Products, and Ultrametric Geometry in Kiselman's Semigroup]{Dynamics, Random Products, and Ultrametric Geometry in Kiselman's Semigroup}

\author{\fnm{Luka} \sur{Andrenšek}}

\abstract{
   \normalfont\unboldmath
   We study certain dynamical and metric aspects of Kiselman's semigroup $K_n$.
   The level function $\mathcal{L}$ is introduced and shown to admit a simple description in terms of right multiplication by generators.
   We show that every sequence of partial products in $K_n$ is eventually constant.
   Using $\mathcal{L}$, we further study sequences of random partial products in $K_n$ and show that, in the independent and identically distributed setting where every generator is chosen with positive probability,
   the hitting time of the eventual constant value is distributed as a sum of $n$ independent geometric random variables.
   Finally, we define a natural ultrametric on $K_n$ arising from the level function and obtain some basic results on the associated metric balls and spheres.
}


\maketitle

%--------------------------
\section{Introduction}

Let $n$ be a natural number.
We define \emph{Kiselman's semigroup} by the following presentation: 
\begin{equation*}
    K_n = \langle a_1, a_2, \dots, a_n \mid a_i^2 = a_i, a_i a_j a_i = a_j a_i a_j = a_i a_j, 1 \leq j < i \leq n \rangle,
\end{equation*}
and denote by $e$ the unit element of $K_n$.
Throughout the paper, we assume $n > 1$.

This family of semigroups arose as a generalisation of a semigroup of operators in convex analysis, which was studied by Kiselman in~\cite{kiselman}.
Many fundamental results on $K_n$ were proved by Kudryavtseva and Mazorchuk in~\cite{kudryavtseva}. 
Kiselman's semigroup arises in graph dynamics, as discussed by Collina and D'Andrea in~\cite{collina}, and admits interesting combinatorial properties in various settings.

D'Andrea and Stella showed in~\cite{dandrea23} that the cardinality of Kiselman's semigroups grows double-exponentially. 
Moreover, we described the endomorphism monoid of $K_n$ in~\cite{andrensek}, and studied certain equations in $K_n$ in~\cite{andrensek2}.

%---------------------------------------------
\section{Main Results}
The central object of this paper is the \emph{level function} $\mathcal{L} : K_n \to \{0, 1, \dots, n\}$, defined in Section~\ref{L_Function}.
The function $\mathcal{L}$ will later be shown to coincide with the distance from the zero element of $K_n$ with respect to a natural ultrametric, and its most important property is the following formula:
\[
\mathcal{L}(x a_i) = 
\begin{cases}
    \mathcal{L}(x) - 1, & i = \mathcal{L}(x), \\
    \mathcal{L}(x), & i \neq \mathcal{L}(x).
\end{cases}
\]
The function $\mathcal{L}$ is the main tool used to study sequences of random partial products, and it also gives rise to the ultrametric structure on $K_n$.

We investigate sequences of partial products in $K_n$, that is, sequences of the form
\[
x_1 x_2 \dots x_j,
\]
where $(x_j)_{j \geq 1}$ is a sequence with values in $\{a_1, a_2, \dots, a_n\}$.
We prove that every such sequence is eventually constant, and we determine its eventual value under certain assumptions.

Using $\mathcal{L}$, we study sequences of random partial products in $K_n$, that is, sequences of random variables of the form 
\[
X_1 X_2 \dots X_j,
\]
where $(X_j)_{j \geq 1}$ is a sequence of independent and identically distributed random variables with values in $\{a_1, a_2, \dots, a_n\}$.
We show that the random variables defined by
\[
\mathcal{L}(X_1 X_2 \dots X_j)
\]
form a Markov chain.
We further study the hitting time at which a sequence of random partial products attains its eventual constant value under the additional assumption that $\mathbb{P}(X_1 = a_i) > 0$ for all $i = 1, 2, \dots, n$.
We show that this hitting time is distributed as a sum of $n$ independent geometric random variables with success probabilities $\mathbb{P}(X_1 = a_i)$ for $i = 1, 2, \dots, n$.
In particular, if $(X_j)_{j \geq 1}$ are uniformly distributed on $\{a_1, a_2, \dots, a_n\}$, then the expected time until the sequence of random partial products becomes constant is $n^2$.

We also define a natural ultrametric on $K_n$, arising from the definition of the level function $\mathcal{L}$, and prove some results about the metric balls and spheres.

The paper is organised as follows.
In Section~\ref{Prelim}, we recall some background and necessary results on $K_n$.
In Section~\ref{DEL}, we define deletion endomorphisms and prove some of their properties.
These endomorphisms are used in the definition of $\mathcal{L}$, which we introduce in Section~\ref{L_Function} and study further there.
In Section~\ref{INF}, we study sequences of partial products and their random counterparts.
In Section~\ref{MET}, we define an ultrametric on $K_n$ and describe some of its basic properties.

%--------------------------
\section{Preliminaries}\label{Prelim}
For a set $X$, denote by $W(X)$ the set of all finite words over $X$ and denote the empty word by $\eq$.
By equipping $W(X)$ with the operation of \emph{concatenation}, $W(X)$ becomes a monoid with unit element $\eq$.
Let $A = \{\aq_1, \aq_2, \dots, \aq_n\}$.
Denote by $\varphi : W(A) \to K_n$ the \emph{canonical epimorphism}, defined by
\[
\varphi(\aq_{i_1} \aq_{i_2} \dots \aq_{i_k}) = a_{i_1} a_{i_2} \dots a_{i_k},
\]
for all $k \geq 1$ and $i_1, i_2, \dots, i_k \in \{1, 2, \dots, n\}$, and $\varphi(\eq) = e$.
In what follows, elements of $W(A)$ are written in boldface, for instance $\wq, \vq, \aq_i$, while elements of $K_n$ are written in standard mathematical italic, for example $x, y, a_i$.
For $k \in \mathbb{N}$, we denote by $[k] = \{1, 2, \dots, k\}$ and also set $[0] = \emptyset$.

In~\cite[Lemma 1]{kudryavtseva}, the following lemma was proved.
\begin{lemma}\label{lemma:shorten}
    The following statements hold.
    \begin{enumerate}[label=(\roman*), ref=\roman*]
        \item\label{lemma:shorten:i} Let $i \in [n]$ and $\wq \in W(\{\aq_1, \dots, \aq_{i-1}\})$.
        Then we have $a_i \varphi(\wq) a_i = a_i \varphi(\wq)$.
        \item\label{lemma:shorten:ii} Let $i \in [n]$ and $\wq \in W(\{\aq_{i+1}, \dots, \aq_n\})$.
        Then we have $a_i \varphi(\wq) a_i = \varphi(\wq) a_i$.
    \end{enumerate}
\end{lemma}

%--------------------
\subsection{Idempotents in $K_n$}\label{IDE}
Let $X \subseteq [n]$.
If $X = \emptyset$, we define $e_X = e$, the unit element of $K_n$, and $\eq_X = \eq$, the empty word in $W(A)$.
Otherwise, write $X = \{i_1, i_2, \dots, i_k\}$, where $i_1 > i_2 > \dots > i_k$, and set
\[
e_X = a_{i_1} a_{i_2} \dots a_{i_k} \quad \text{and} \quad \eq_X = \aq_{i_1} \aq_{i_2} \dots \aq_{i_k}.
\]
We then have $\varphi(\eq_X) = e_X$ for every $X \subseteq [n]$.
In~\cite[Remark 16]{kudryavtseva}, it was observed that $e_{[n]}$ is the zero element of $K_n$ and we denote it by $f$.
Moreover, in~\cite[Proposition 11]{kudryavtseva}, it was shown that the set $\{e_X \mid X \subseteq [n]\}$ is the set of all idempotents in $K_n$.

Define the \emph{content} map $c : K_n \to \mathbbm{2}^{[n]}$, where $c(x)$ is the set of all $i$ such that $\aq_i$ occurs in $\wq$ for any $\wq \in W(A)$ with $x = \varphi(\wq)$.
In~\cite[Lemma 10]{kudryavtseva}, it was shown that $c$ is well-defined and that it is a semigroup epimorphism when $\mathbbm{2}^{[n]}$ is equipped with the operation $\cup$.

In~\cite[Lemma 12]{kudryavtseva}, the following lemma was proved.
\begin{lemma}\label{idempotent:power}
    Let $x \in K_n$. Then $x^k = e_{c(x)}$ for all $k \geq |c(x)|$.
\end{lemma}

%------------------------------
\subsection{Morphisms of $K_n$}

In~\cite[Proposition~20]{kudryavtseva}, the following was proved.
\begin{proposition}
    The only automorphism of $K_n$ is the identity. The map $a_i \mapsto a_{n-i+1}$ extends uniquely to an 
    antiautomorphism of $K_n$. This is the only antiautomorphism of $K_n$.
\end{proposition}
We denote the unique antiautomorphism of $K_n$ by $\tau$.

Let $A \in \mathbb{R}^{2 \times 2}$ and $M \in \mathbb{R}^{k_1 \times k_2}$ be real matrices with $k_1, k_2 \geq 2$.
The matrix $A$ is said to be a \emph{$2 \times 2$ submatrix} of $M$ if there exist rows
$x, y \in [k_1]$ and columns $i, j \in [k_2]$ with $x < y$ and $i < j$ such that 
\[
    A = 
    \begin{pmatrix}
        M_{x, i} & M_{x, j} \\
        M_{y, i} & M_{y, j}
    \end{pmatrix},
\]
where $M_{x, i}$ denotes the entry of $M$ in row $x$ and column $i$.
Let $P$ be the following $2 \times 2$ matrix:
\begin{equation}
    P = 
    \begin{pmatrix}
        0 & 1 \\
        1 & 0
    \end{pmatrix}.
\end{equation}
In~\cite{andrensek}, we introduced the following set for $n \geq 2$:
\[
    D_n = \{ M \in \{0, 1\}^{n \times n} \mid P \text{ is not a } 2 \times 2 \text{ submatrix of } M \}.
\]
The binary operation $\cdot$ on $D_n$ is defined as follows.
Let $A, B \in D_n$.
Then $C = A \cdot B$ is given by
\begin{align*}
   C_{i, j} = \bigvee_{k=1}^n (A_{i, k} \land B_{k, j}).
\end{align*}
Then $D_n$ becomes a monoid with unit element being the identity matrix $I$.
Here, $\lor$ denotes the Boolean join and $\land$ the Boolean meet on $\{0, 1\}$.

In~\cite[Theorem~1]{andrensek}, we proved that the endomorphism monoid $\mathrm{End}(K_n)$ is isomorphic to $D_n$.
Here, we denote this isomorphism by $\Theta : \mathrm{End}(K_n) \to D_n$.
We showed that for any endomorphism $\psi \in \mathrm{End}(K_n)$, there exist sets $X_1, X_2, \dots, X_n \subseteq [n]$ such that
\begin{equation*}
    \psi(a_i) = e_{X_i},
\end{equation*}
for all $i = 1, 2, \dots, n$.
Then $\Theta(\psi)$ is a Boolean matrix where $\Theta(\psi)_{x, i} = 1$ if and only if $x \in X_i$ for $x, i \in [n]$.

In~\cite[Section 7]{kudryavtseva}, the authors defined the \emph{height} $h(x) \in \mathbb{N}_0$ for $x \in K_n$ and proved the following lemma~\cite[Lemma~23]{kudryavtseva}.
\begin{lemma}\label{lemma:height}
    Let $x, y \in K_n$ be such that $x y \neq y$.
    Then $h(x y) < h(y)$.
\end{lemma}

%--------------------------------------------
\subsection{Solutions to Equations in $K_n$}
In this subsection, we recall some results from~\cite{andrensek2} regarding solutions to equations in $K_n$.
We proved the following statement.
\begin{proposition}\label{ct2}
    Let $x \in K_n$ and let $k \in \{2, 3, \dots, n\}$ and $r \in \{1, 2, \dots, n-1\}$.
    If $x a_k = f$ or $a_r x = f$, then $x = f$. 
\end{proposition}

We defined the set $R$ by 
\[
R = \{x \in K_n \mid x a_1 = f\},
\]
and proved the following.
\begin{proposition}\label{bc1}
   We have $|R| = 1 + |K_{n-1}|$.
\end{proposition}

We defined the map $m : K_n \to \{0, 1, \dots, n\}$ by
\begin{equation}\label{m_definition}
    m(x) = \min \{i \in \{0, 1, \dots, n\} \mid x e_{[i]} = f\},
\end{equation}
and proved the following theorem. 
\begin{theorem}
    We have that 
    \begin{equation*}
        R = \{e_{\{2, 3, \dots, n\}}\} \cup \{x a_1 e_{\{2, 3, \dots, m(x)\}} \mid x \in \langle a_2, a_3, \dots, a_n \rangle\}.
    \end{equation*}
\end{theorem}

%-------------------------------------------------------------------
\section{Deletion Endomorphisms}\label{DEL}
In this section, we define deletion endomorphisms and prove several properties that play an important role in Section~\ref{L_Function}.

\begin{definition}\label{def:31}
    Let $X \subseteq [n]$.
    We define the \emph{deletion map} associated with $X$ to be the map $\partial_X : W(A) \to W(A)$ given as follows.
    If $\wq = \aq_{i_1} \aq_{i_2} \dots \aq_{i_k}$, then
    \[
    \partial_X(\wq) = \eq_{\{i_1\} \setminus X} \eq_{\{i_2\} \setminus X} \dots \eq_{\{i_k\} \setminus X},
    \]
    and $\partial_X(\eq) = \eq$.
    In other words, $\partial_X$ deletes from a word all letters $\aq_i$ with $i \in X$.
\end{definition}
It is easy to verify that $\partial_X$ is an endomorphism of $W(A)$.

Let $i \in [n]$ and denote by $I^{(i)}$ the $n \times n$ Boolean matrix with entry $1$ in row $i$ and column $i$, and $0$ elsewhere.
For $X \subseteq [n]$, we define
\begin{equation*}
    I_X = \sum_{i\in X} I^{(i)}.
\end{equation*}
Since $I_X$ is a diagonal Boolean matrix, it lies in $D_n$.

\begin{definition}
    Let $X \subseteq [n]$.
    We define the \emph{deletion endomorphism} associated with $X$ to be the endomorphism $\overline{\partial}_X : K_n \to K_n$, defined by
    \begin{equation*}
        \overline{\partial}_X = \Theta^{-1}(I_{X^c}).
    \end{equation*}
\end{definition}

Let $X \subseteq [n]$.
Since $\overline{\partial}_X$ is an endomorphism of $K_n$, there exist sets $X_1, X_2, \dots, X_n \subseteq [n]$ such that 
\begin{equation*}
    \overline{\partial}_X(a_i) = e_{X_i},
\end{equation*}
for all $i = 1, 2, \dots, n$.
By the definition of $\Theta$, we have $\Theta(\overline{\partial}_X)_{x, i} = 1$ if and only if $x \in X_i$.
Since $\Theta(\overline{\partial}_X) = I_{X^c}$, we have $x \in X_i$ if and only if $(I_{X^c})_{x, i} = 1$, which is further equivalent to $x = i$ and $x \notin X$.
Hence $X_i = \{i\} \setminus X$.
Therefore, for all $i$, we have
\begin{equation}\label{100}
    \overline{\partial}_X(a_i) = e_{\{i\} \setminus X} = 
    \begin{cases}
        a_i, & i \notin X, \\
        e, & i \in X.
    \end{cases}
\end{equation}
Hence, the maps $\partial_X$ and $\overline{\partial}_X$ are closely related.
We now establish several properties of the deletion endomorphisms analogous to those holding for the maps $\partial_X$.

\begin{lemma}\label{lemma:deletion-phi}
    Let $X \subseteq [n]$ and $x \in K_n$.
    If $x = \varphi(\wq)$, where $\wq \in W(A)$, then $\overline{\partial}_X(x) = \varphi(\partial_X(\wq))$.
\end{lemma}

\begin{proof}
    If $\wq = \eq$, then $x = e$, and the conclusion follows.
    Otherwise, write $\wq = \aq_{i_1} \aq_{i_2} \dots \aq_{i_k}$ and hence 
    $x = a_{i_1} a_{i_2} \dots a_{i_k}$.
    Then we have
    \[
    \varphi(\partial_X(\wq)) = \varphi(\eq_{\{i_1\} \setminus X} \eq_{\{i_2\} \setminus X} \dots \eq_{\{i_k\} \setminus X})
    = e_{\{i_1\} \setminus X} e_{\{i_2\} \setminus X} \dots e_{\{i_k\} \setminus X}.
    \]
    On the other hand, by~\eqref{100} and the fact that $\overline{\partial}_X$ is an endomorphism, we have 
    \[
    \overline{\partial}_X(x) = \overline{\partial}_X(a_{i_1} a_{i_2} \dots a_{i_k}) =
    e_{\{i_1\} \setminus X} e_{\{i_2\} \setminus X} \dots e_{\{i_k\} \setminus X}.
    \]
    This concludes the proof.
\end{proof}

\begin{lemma}\label{lemma:0}
    Let $X, Y \subseteq [n]$.
    Then $\overline{\partial}_X(e_Y) = e_{Y \setminus X}$.
\end{lemma}

\begin{proof}
    The claim follows by applying Lemma~\ref{lemma:deletion-phi} with $x = e_Y$ and $\wq = \eq_Y$, and using the fact that $\partial_X(\eq_Y) = \eq_{Y \setminus X}$.
\end{proof}

For $X, Y \subseteq [n]$, we write $\overline{\partial}_X \overline{\partial}_Y$ and $\partial_X \partial_Y$ instead of $\overline{\partial}_X \circ \overline{\partial}_Y$ and $\partial_X \circ \partial_Y$, respectively.
For $i \in [n]$ we also write $\overline{\partial}_i$ and $\partial_i$ instead of $\overline{\partial}_{\{i\}}$ and $\partial_{\{i\}}$, respectively.

\begin{proposition}\label{del_prop}
    Let $X, Y \subseteq [n]$.
    Then $\overline{\partial}_{X \cup Y} = \overline{\partial}_X \overline{\partial}_Y$.
\end{proposition}

\begin{proof}
    Since $\Theta$ is an isomorphism, it suffices to prove that $I_{(X \cup Y)^c} = I_{X^c} \cdot I_{Y^c}$.
    Since $(X \cup Y)^c = X^c \cap Y^c$, it suffices to show that for any sets 
    $A, B \subseteq [n]$, we have $I_{A \cap B} = I_A \cdot I_B$. The latter equality is immediate.
\end{proof}
In particular, Proposition~\ref{del_prop} implies that the deletion endomorphisms commute.

We now prove some technical results on the endomorphisms $\overline{\partial}_{[m]}$ for $m \in \{0, 1, \dots, n\}$.
These results will be crucial in proving properties of the level function $\mathcal{L}$ in Section~\ref{L_Function}.
For $m \in [n]$, by Proposition~\ref{del_prop}, we have $\overline{\partial}_{[m]} = \overline{\partial}_{[m-1]} \overline{\partial}_m = \overline{\partial}_m \overline{\partial}_{[m-1]}$.
Also note that $\overline{\partial}_{[0]}(x) = x$ for any $x \in K_n$.

\begin{lemma}\label{lf1}
    Let $x \in K_n$, $m \in \{0, 1, 2, \dots, n\}$, and $r \in \{0, 1, \dots, n - m\}$.
    If $\overline{\partial}_{[m]}(x) = e_{[n] \setminus [m]}$, then 
    \begin{equation*}
        \overline{\partial}_{[m+r]}(x) = e_{[n] \setminus [m+r]}.
    \end{equation*}
\end{lemma}

\begin{proof}
    We proceed with induction on $r$.
    The base case $r = 0$ holds by the assumption of the lemma.
    Assume that $r \in \{1, 2, \dots, n - m\}$ and $\overline{\partial}_{[m + r -1]}(x) = e_{[n] \setminus [m + r - 1]}$.
    Then we have 
    \begin{align*}
        \overline{\partial}_{[m+r]}(x) &= \overline{\partial}_{m+r}(\overline{\partial}_{[m + r - 1]}(x)) \\
        &= \overline{\partial}_{m+r}(e_{[n] \setminus [m+r-1]}) \\
        &= e_{([n] \setminus [m + r -1]) \setminus \{m+r\}} \\
        &= e_{[n] \setminus [m+r]},
    \end{align*}
    where we used Lemma~\ref{lemma:0} on the second to last step.
    This proves the induction step and the proof is finished.
\end{proof}

\begin{lemma}\label{lemma_chain:1}
    Let $x \in K_n$ and $m \in [n]$.
    Then we have
    \[
    \overline{\partial}_{[m-1]}(x) a_m = \overline{\partial}_{[m]}(x) a_m.
    \]
\end{lemma}

\begin{proof}
    We write $x = \varphi(\wq)$ for some $\wq \in W(A)$.
    By Lemma~\ref{lemma:deletion-phi}, we have 
    \[
    \overline{\partial}_{[m-1]}(x) = \varphi(\partial_{[m-1]}(\wq)).
    \]
    Set $\vq = \partial_{[m-1]}(\wq)$ and observe that $\vq \in W(\{\aq_m, \aq_{m+1}, \dots, \aq_n\})$.
    We obtain
    \[
    \overline{\partial}_{[m-1]}(x) a_m = \varphi(\vq) a_m.
    \]
    If $\aq_m$ does not occur in $\vq$, we have $\vq = \partial_m(\vq)$.
    If $\aq_m$ occurs in $\vq$, then $\aq_m$ is the letter of minimal index occurring in $\vq$.
    Lemma~\ref{lemma:shorten}~(\ref{lemma:shorten:ii}) implies that we can delete all occurrences of $a_m$ in $\varphi(\vq) a_m$, except for the rightmost $a_m$.
    Hence we get
    \[
    \varphi(\vq) a_m = \overline{\partial}_m(\varphi(\vq)) a_m = \varphi(\partial_m(\vq)) a_m.
    \]

    Thus, in both cases, we have $\varphi(\vq) a_m = \varphi(\partial_m(\vq)) a_m$, and hence
    \begin{align*}
        \overline{\partial}_{[m-1]}(x) a_m &= \varphi(\vq) a_m \\
        &= \varphi(\partial_m(\vq)) a_m \\
        &= \overline{\partial}_m(\varphi(\vq)) a_m \\
        &= \overline{\partial}_m(\overline{\partial}_{[m-1]}(x)) a_m \\
        &= \overline{\partial}_{[m]}(x) a_m.
    \end{align*}
    This completes the proof.
\end{proof}

\begin{lemma}\label{lemma_chain:2}
    Let $x \in K_n$ and $i \in \{0, 1, \dots, n\}$.
    Then 
    \[
    x e_{[i]} = \overline{\partial}_{[j]}(x) e_{[i]},
    \]
    for all $j \in \{0, 1, \dots, i\}$.
\end{lemma}

\begin{proof}
    We proceed with a proof by induction on $j$.
    The base case $j = 0$ is obvious since $\overline{\partial}_{[0]}(x) = x$. 
    Now assume that $j \in \{1, 2, \dots, i\}$ and 
    \[
    x e_{[i]} = \overline{\partial}_{[j-1]}(x) e_{[i]}.
    \]
    Since 
    \[
    e_{[i]} = e_{[i] \setminus [j]} a_j e_{[j-1]},
    \]
    we have 
    \begin{align*}
        x e_{[i]} &= \overline{\partial}_{[j-1]}(x) e_{[i]} \\
                  &= \overline{\partial}_{[j-1]}(x) e_{[i] \setminus [j]} a_j e_{[j-1]}.
    \end{align*}
    Since $\overline{\partial}_{[j-1]}(e_{[i] \setminus [j]}) = e_{[i] \setminus [j]}$, we further have 
    \begin{align*}
        x e_{[i]} &= \overline{\partial}_{[j-1]}(x) e_{[i] \setminus [j]} a_j e_{[j-1]} \\
                  &= \overline{\partial}_{[j-1]}(x) \overline{\partial}_{[j-1]}(e_{[i] \setminus [j]}) a_j e_{[j-1]} \\
                  &= \overline{\partial}_{[j-1]}(x e_{[i] \setminus [j]}) a_j e_{[j-1]}.
    \end{align*}
    Since $j \in \{1, 2, \dots, i\} \subseteq [n]$, Lemma~\ref{lemma_chain:1} implies 
    \begin{align*}
        x e_{[i]} &= \overline{\partial}_{[j-1]}(x e_{[i] \setminus [j]}) a_j e_{[j-1]} \\
                  &= \overline{\partial}_{[j]}(x e_{[i] \setminus [j]}) a_j e_{[j-1]} \\ 
                  &= \overline{\partial}_{[j]}(x) \overline{\partial}_{[j]}(e_{[i] \setminus [j]}) a_j e_{[j-1]} \\
                  &= \overline{\partial}_{[j]}(x) e_{[i] \setminus [j]} a_j e_{[j-1]} \\
                  &= \overline{\partial}_{[j]}(x) e_{[i]},
    \end{align*}
    which proves the induction step and completes the proof.
\end{proof}

\begin{lemma}\label{lf2}
    Let $x \in K_n$ and $m \in [n]$.
    If $\overline{\partial}_{[m]}(x) = e_{[n] \setminus [m]}$, then
    \[
    \overline{\partial}_{[m-1]}(x a_{m}) = e_{[n] \setminus [m-1]}.
    \]
\end{lemma}

\begin{proof}
    By~\eqref{100}, we have $\overline{\partial}_{[m-1]}(a_m) = a_m$, and hence
    \[
    \overline{\partial}_{[m-1]}(x a_m) = \overline{\partial}_{[m-1]}(x) a_m.
    \]
    Lemma~\ref{lemma_chain:1} further implies
    \begin{align*}
        \overline{\partial}_{[m-1]}(x) a_m &= \overline{\partial}_{[m]}(x) a_m \\
        &= e_{[n] \setminus [m]} a_m \\
        &= e_{[n] \setminus [m-1]}.
    \end{align*}
    This concludes the proof.
\end{proof}

\begin{lemma}\label{L3}
    Let $x \in K_n$, $m \in \{0, 1, \dots, n - 2\}$, $k \in \{2, 3, \dots, n - m\}$, and $r \in \{1, 2, \dots, n - m - 1\}$.
    If
    \[
    \overline{\partial}_{[m]}(x a_{m+k}) = e_{[n] \setminus [m]} \quad \text{or} \quad \overline{\partial}_{[m]}(a_{m+r} x) = e_{[n] \setminus [m]},
    \]
    then
    \[
    \overline{\partial}_{[m]}(x) = e_{[n] \setminus [m]}.
    \]
\end{lemma}

\begin{proof}
    Since $\overline{\partial}_{[m]}(a_{m+k}) = a_{m+k}$ and $\overline{\partial}_{[m]}(a_{m+r}) = a_{m+r}$ by~\eqref{100}, we get
    \begin{equation}\label{000}
        \overline{\partial}_{[m]}(x) a_{m+k} = e_{[n] \setminus [m]} \quad \text{or} \quad a_{m+r} \overline{\partial}_{[m]}(x) = e_{[n] \setminus [m]}.
    \end{equation}

    Denote by $S = \langle a_{m+1}, \dots, a_n \rangle$ the subsemigroup of $K_n$ generated by $\{a_{m+1}, \dots, a_n\}$.
    Then $S$ is isomorphic to $K_{n-m}$.
    Let $\phi : S \to K_{n-m}$ be the semigroup isomorphism such that $\phi(a_i) = a_{i - m}$ for $i = m+1, \dots, n$.

    Write $x = \varphi(\wq)$ for some $\wq \in W(A)$.
    By Lemma~\ref{lemma:deletion-phi}, we have $\overline{\partial}_{[m]}(x) = \varphi(\partial_{[m]}(\wq)) \in S$, and we also have $a_{m+k}, a_{m+r}, e_{[n] \setminus [m]} \in S$.
    Applying $\phi$ to~\eqref{000}, and using the fact that $\phi$ is an isomorphism and $\phi(e_{[n] \setminus [m]}) = e_{[n-m]}$, we obtain
    \[
    \phi(\overline{\partial}_{[m]}(x)) a_k = e_{[n-m]} \quad \text{or} \quad a_r \phi(\overline{\partial}_{[m]}(x)) = e_{[n-m]}.
    \]
    Since $k \in \{2, 3, \dots, n - m\}$ and $r \in \{1, 2, \dots, n-m-1\}$, we may apply Proposition~\ref{ct2} in $K_{n-m}$ to conclude that
    \[
    \phi(\overline{\partial}_{[m]}(x)) = e_{[n-m]}.
    \]
    Since $\phi$ is injective and $\phi(e_{[n] \setminus [m]}) = e_{[n-m]}$, we get
    \[
    \overline{\partial}_{[m]}(x) = e_{[n] \setminus [m]}.
    \]
    This completes the proof.
\end{proof}

%-------------------------------------
\section{The Level Function}\label{L_Function}
In this section, we define the level function and study its properties.

\begin{definition}
    The \emph{level function} $\mathcal{L} : K_n \to \{0, 1, \dots, n\}$ is defined by
    \[
    \mathcal{L}(x) = \min \{i \in \{0, 1, \dots, n\} \mid \overline{\partial}_{[i]}(x) = e_{[n] \setminus [i]} \}.
    \]
\end{definition}

Let $x \in K_n$ and $x = \varphi(\wq)$ with $\wq \in W(A)$.
Since
\[
\overline{\partial}_{[n]}(x) = \varphi(\partial_{[n]}(\wq)) = \varphi(\eq) = e = e_{\emptyset} = e_{[n] \setminus [n]},
\]
the set 
\[
\{i \in \{0, 1, \dots, n\} \mid \overline{\partial}_{[i]}(x) = e_{[n] \setminus [i]} \}
\]
contains $n$, and is therefore nonempty.
Therefore, $\mathcal{L}(x)$ is well-defined.

Intuitively, $\mathcal{L}$ measures how close an element is to the zero element $f$.
In Section~\ref{MET}, we define an ultrametric $d$ on $K_n$ in~\eqref{Kn_Metric}, and show that, for $x \in K_n$, we have
\[
\mathcal{L}(x) = d(x, f).
\]
We focus on $d(x, f)$ rather than $d(x, y)$ for general $y \in K_n$, since $\mathcal{L}(x)$ admits properties, such as Theorem~\ref{RL_formula}, that do not hold for the map $x \mapsto d(x, y)$ for arbitrary $y \in K_n$.

\begin{lemma}\label{RU1}
    Let $x \in K_n$. Then $\mathcal{L}(x) = 0$ if and only if $x = f$.
\end{lemma}

\begin{proof}
    If $\mathcal{L}(x) = 0$, then
    \[
    x = \overline{\partial}_{[0]}(x) = e_{[n] \setminus [0]} = f.
    \]
    Conversely, we have
    \[
    \overline{\partial}_{[0]}(f) = f = e_{[n] \setminus [0]},
    \]
    and hence $\mathcal{L}(f) = 0$.
\end{proof}

\begin{corollary}\label{RL2}
    Let $x \in K_n$ and let $k \in \{\mathcal{L}(x), \mathcal{L}(x) + 1, \dots, n\}$. Then 
    \[
    \overline{\partial}_{[k]}(x) = e_{[n] \setminus [k]}.
    \]
\end{corollary}

\begin{proof}
    Set $l = \mathcal{L}(x)$.
    Then $\overline{\partial}_{[l]}(x) = e_{[n] \setminus [l]}$.
    Lemma~\ref{lf1} implies that for any $r \in \{0, 1, \dots, n-l\}$, we have
    \[
    \overline{\partial}_{[l + r]}(x) = e_{[n] \setminus [l + r]}.
    \]
    Hence, for any $k \in \{l, l+1, \dots, n \}$, we obtain
    \[
    \overline{\partial}_{[k]}(x) = e_{[n] \setminus [k]}.
    \]
\end{proof}

\begin{lemma}\label{L_equal_n}
    Let $x \in K_n$. Then $\mathcal{L}(x) = n$ if and only if $c(x) \subseteq [n-1]$.
\end{lemma}

\begin{proof}
    Write $x = \varphi(\wq)$ for some $\wq \in W(A)$.
    If $\mathcal{L}(x) = n$, then by the definition of $\mathcal{L}$,  $\overline{\partial}_{[n-1]}(x) \neq e_{[n] \setminus [n-1]} = a_n$, and hence
    \[
    \varphi(\partial_{[n-1]}(\wq)) \neq a_n.
    \]
    If $\aq_n$ occurs in $\wq$, it follows that $\varphi(\partial_{[n-1]}(\wq)) = a_n$, which is a contradiction, therefore $\aq_n$ does not occur in $\wq$.
    Hence $c(x) \subseteq [n-1]$.
    Conversely, if $c(x) \subseteq [n-1]$, then 
    \[
    \overline{\partial}_{[n-1]}(x) = \varphi(\partial_{[n-1]}(\wq)) = \varphi(\eq) = e \neq a_n = e_{[n] \setminus [n-1]}.
    \]
    Corollary~\ref{RL2} then implies $\mathcal{L}(x) > n - 1$, and hence $\mathcal{L}(x) = n$.
\end{proof}

\begin{proposition}\label{L:upper:bound}
    Let $x, y \in K_n$. Then 
    \[
    \mathcal{L}(x y) \leq \min \{\mathcal{L}(x), \mathcal{L}(y)\}.
    \]
\end{proposition}

\begin{proof}
    Denote $l_1 = \mathcal{L}(x)$ and $l_2 = \mathcal{L}(y)$.
    Since $c(\overline{\partial}_{[l_1]}(y)) \subseteq [n] \setminus [l_1]$, Lemma~\ref{lemma:shorten}~(\ref{lemma:shorten:i}) implies
    \begin{align*}
        \overline{\partial}_{[l_1]}(x y) &= \overline{\partial}_{[l_1]}(x) \overline{\partial}_{[l_1]}(y) \\
        &= e_{[n] \setminus [l_1]} \overline{\partial}_{[l_1]}(y) \\
        &= e_{[n] \setminus [l_1]},
    \end{align*}
    which proves $\mathcal{L}(x y) \leq l_1$.
    Since $c(\overline{\partial}_{[l_2]}(x)) \subseteq [n] \setminus [l_2]$, Lemma~\ref{lemma:shorten}~(\ref{lemma:shorten:ii}) implies
    \begin{align*}
        \overline{\partial}_{[l_2]}(x y) &= \overline{\partial}_{[l_2]}(x) \overline{\partial}_{[l_2]}(y) \\
        &= \overline{\partial}_{[l_2]}(x) e_{[n] \setminus [l_2]} \\
        &= e_{[n] \setminus [l_2]},
    \end{align*}
    which proves $\mathcal{L}(x y) \leq l_2$.
    This concludes the proof.
\end{proof}

\begin{proposition}\label{RP1}
    Let $x \in K_n$. If $\mathcal{L}(x) > 0$, then 
    \[
    \mathcal{L}(x a_{\mathcal{L}(x)}) = \mathcal{L}(x) - 1.
    \]
\end{proposition}

\begin{proof}
    We denote $l = \mathcal{L}(x)$.
    Since $\overline{\partial}_{[l]}(x) = e_{[n] \setminus [l]}$ and $l \in [n]$, Lemma~\ref{lf2} implies that
    \[
    \overline{\partial}_{[l-1]}(x a_l) = e_{[n] \setminus [l-1]},
    \]
    and therefore $\mathcal{L}(x a_l) \leq l - 1$.

    If $l = 1$, then $\mathcal{L}(x a_l) \leq 0$, and hence $\mathcal{L}(x a_l) = 0 = \mathcal{L}(x) - 1$. 

    Now assume that $l \geq 2$.
    Suppose that $\mathcal{L}(x a_l) \leq l -2$.
    By Corollary~\ref{RL2}, we obtain
    \[
    \overline{\partial}_{[l - 2]}(x a_l) = e_{[n] \setminus [l - 2]}.
    \]
    Since $l - 2 \in \{0, 1, \dots, n - 2\}$, Lemma~\ref{L3} for $m = l - 2$ and $k = 2$ implies that 
    \[
    \overline{\partial}_{[l-2]}(x) = e_{[n] \setminus [l -2]}.
    \]
    Hence
    \[
    \mathcal{L}(x) \leq l - 2 = \mathcal{L}(x) - 2,
    \]
    which is a contradiction.
    Hence we must have $\mathcal{L}(x a_l) \geq l - 1$.
    Therefore, we indeed have
    \[
    \mathcal{L}(x a_{\mathcal{L}(x)}) = \mathcal{L}(x) - 1.
    \]
    This concludes the proof.
\end{proof}

\begin{proposition}\label{RP2}
    Let $x \in K_n$ with $\mathcal{L}(x) > 0$ and let $i \in \{1, 2, \dots, \mathcal{L}(x) - 1\}$.
    Then 
    \[
    \mathcal{L}(x a_i) = \mathcal{L}(x).
    \]
\end{proposition}

\begin{proof}
    We denote $l = \mathcal{L}(x)$.
    If $l = 1$, then the set $\{1, 2, \dots, l-1\}$ is empty and the claim vacuously holds.
    Now assume $l \geq 2$.
    By Proposition~\ref{L:upper:bound}, we have $\mathcal{L}(x a_i) \leq l$.

    Suppose $\mathcal{L}(x a_i) \leq l - 1$. 
    By Corollary~\ref{RL2}, we have 
    \[
    \overline{\partial}_{[l-1]}(x a_i) = e_{[n] \setminus [l-1]}.
    \]
    Since $i \in \{1, 2, \dots, l -1\}$, we have $\overline{\partial}_{[l-1]}(a_i) = e$, by~\eqref{100}.
    It follows that
    \[
    \overline{\partial}_{[l-1]}(x) = \overline{\partial}_{[l-1]}(x) \overline{\partial}_{[l-1]}(a_i) = \overline{\partial}_{[l-1]}(x a_i) = e_{[n] \setminus [l-1]}.
    \]
    This implies
    \[
    \mathcal{L}(x) \leq l -1 = \mathcal{L}(x) - 1,
    \]
    which is a contradiction.
    Hence we must have $\mathcal{L}(x a_i) \geq l$, and therefore 
    \[
    \mathcal{L}(x a_i) = \mathcal{L}(x).
    \]
\end{proof}

\begin{proposition}\label{RP3}
    Let $x \in K_n$ with $\mathcal{L}(x) > 0$ and let $i \in \{\mathcal{L}(x) + 1, \mathcal{L}(x) + 2, \dots, n\}$.
    Then
    \[
    \mathcal{L}(x a_i) = \mathcal{L}(x).
    \]
\end{proposition}

\begin{proof}
    We denote $l = \mathcal{L}(x)$. 
    If $l = n$, there does not exist any $i \in \{l+1, l+2, \dots, n\}$, and the statement is vacuously true.
    Now assume that $l \leq n-1$.
    By Proposition~\ref{L:upper:bound}, we have $\mathcal{L}(x a_i) \leq l$.

    Suppose that $\mathcal{L}(x a_i) \leq l - 1$.
    By Corollary~\ref{RL2}, we have
    \[
    \overline{\partial}_{[l-1]}(x a_i) = e_{[n] \setminus [l-1]}.
    \]
    Since $l-1 \in \{0, 1, \dots, n - 2\}$ and $i - l + 1 \in \{2, 3, \dots, n-l+1\}$, we use Lemma~\ref{L3} for $m = l - 1$ and $k = i - l + 1$ to obtain
    \[
    \overline{\partial}_{[l-1]}(x) = e_{[n] \setminus [l-1]}.
    \]
    Hence
    \[
    \mathcal{L}(x) \leq l - 1 = \mathcal{L}(x) - 1,
    \]
    which is a contradiction, and hence we must have $\mathcal{L}(x a_i) \geq l$.
    Thus we have
    \[
    \mathcal{L}(x a_i) = \mathcal{L}(x).
    \]
\end{proof}

\begin{proposition}\label{propo}
    Let $x \in K_n$.
    If $\mathcal{L}(x) = 0$, then for any $y \in K_n$, we have
    \[
    \mathcal{L}(xy) = \mathcal{L}(yx) = \mathcal{L}(x).
    \]
\end{proposition}

\begin{proof}
    By Lemma~\ref{RU1}, we have $x = f$.
    Since $x$ is the zero element, we have $xy = yx = x$ for any $y \in K_n$, and the result follows.
\end{proof}

The following formula tells us how $\mathcal{L}$ behaves under right multiplication by the generators $a_i$.
It is its most important property.
\begin{theorem}\label{RL_formula}
    Let $x \in K_n$ and $i \in [n]$.
    Then 
    \begin{equation}\label{formula}
        \mathcal{L}(x a_i) = 
        \begin{cases}
            \mathcal{L}(x) - 1, & i = \mathcal{L}(x), \\
            \mathcal{L}(x), &  i \neq \mathcal{L}(x).
        \end{cases}
    \end{equation}
\end{theorem}

\begin{proof}
    If $\mathcal{L}(x) > 0$, Propositions~\ref{RP2} and~\ref{RP3} imply that
    \[
    \mathcal{L}(x a_i) = \mathcal{L}(x),
    \]
    for all $i \in [n] \setminus \{\mathcal{L}(x)\}$.
    By Proposition~\ref{RP1}, we have
    \[
    \mathcal{L}(x a_{\mathcal{L}(x)}) = \mathcal{L}(x) - 1,
    \]
    which proves the validity of~\eqref{formula} in this case.

    If $\mathcal{L}(x) = 0$, then $i \neq \mathcal{L}(x)$ for all $i \in [n]$.
    We use Proposition~\ref{propo} for $y = a_i$ to conclude that~\eqref{formula} also holds in this case.
\end{proof}

Theorem~\ref{RL_formula} describes the behaviour of $\mathcal{L}$ under right multiplication by generators.
The following proposition gives the corresponding rule for left multiplication by generators $a_i$ with $i < n$.
\begin{proposition}\label{L_leftmul}
    Let $x \in K_n$ and $i \in [n-1]$. Then
    \[
    \mathcal{L}(a_i x) = \mathcal{L}(x).
    \]
\end{proposition}

\begin{proof}
    Set $l = \mathcal{L}(x)$. 
    If $l = 0$, the conclusion follows by Proposition~\ref{propo} with $y = a_i$.
    If $l = n$, then by Lemma~\ref{L_equal_n}, we have $c(x) \subseteq [n-1]$.
    Hence we have $c(a_i x) \subseteq [n-1]$, and again by Lemma~\ref{L_equal_n}, we conclude
    \[
    \mathcal{L}(a_i x) = n = \mathcal{L}(x).
    \]

    Now we assume that $1 \leq l \leq n-1$.
    Proposition~\ref{L:upper:bound} implies $\mathcal{L}(a_i x) \leq l$. 

    Suppose $\mathcal{L}(a_i x) \leq l - 1$.
    Then by Corollary~\ref{RL2}, we have 
    \begin{equation}\label{bbb}
        \overline{\partial}_{[l-1]}(a_i x) = e_{[n] \setminus [l-1]}.
    \end{equation}

    If $i \leq l - 1$, then by~\eqref{100}, $\overline{\partial}_{[l-1]}(a_i) = e$, and hence from~\eqref{bbb} we obtain
    \[
    \overline{\partial}_{[l-1]}(x) = e_{[n] \setminus [l-1]},
    \]
    which implies
    \[
    \mathcal{L}(x) \leq l -1 = \mathcal{L}(x) - 1,
    \]
    which is a contradiction.

    Now assume that $i \geq l$.
    Since $l-1 \in \{0, 1, \dots, n-2\}$ and $i-l+1 \in \{1, 2, \dots, n-l\}$, we use Lemma~\ref{L3} for $m=l-1$ and $r=i-l+1$ to conclude from~\eqref{bbb} that
    \[
    \overline{\partial}_{[l-1]}(x) = e_{[n] \setminus [l-1]},
    \]
    which again leads to
    \[
    \mathcal{L}(x) \leq l - 1 = \mathcal{L}(x) - 1,
    \]
    a contradiction.

    Hence, we reached a contradiction in both cases, and therefore we must have $\mathcal{L}(a_i x) \geq l$, and therefore 
    \[
    \mathcal{L}(a_i x) = \mathcal{L}(x).
    \]
\end{proof}

\begin{theorem}
    Let $x, y \in K_n$ with $c(y) \subseteq [n-1]$. Then
    \[
    \mathcal{L}(yx) = \mathcal{L}(x).
    \]
\end{theorem}

\begin{proof}
    If $y = e$, the conclusion follows.
    Otherwise, write $y = a_{i_1} a_{i_2} \dots a_{i_k}$ where
    $i_j \in [n-1]$ for all $j \in [k]$. 
    By Proposition~\ref{L_leftmul}, we have
    \[
    \mathcal{L}(x) = \mathcal{L}(a_{i_k} x) = \mathcal{L}(a_{i_{k-1}} a_{i_k} x) = \cdots = \mathcal{L}(a_{i_1} a_{i_2} \dots a_{i_k} x) = \mathcal{L}(yx).
    \]
\end{proof}

Lemma~\ref{L_equal_n} implies that for every $j \in \{0, 1, \dots, n-1\}$, we have
\[
\mathcal{L}(e_{[n-1] \setminus [j]}) = n.
\]
Since
\[
\overline{\partial}_{[i]}(e_{[n] \setminus [j]}) = e_{([n] \setminus [j]) \setminus [i]} = e_{[n] \setminus [\max\{i, j\}]},
\]
for $i \in \{0, 1, \dots, n\}$, we obtain
\[
\mathcal{L}(a_n e_{[n-1] \setminus [j]}) = \mathcal{L}(e_{[n] \setminus [j]}) = \min \{ i \in \{0, 1, \dots, n\} \mid e_{[n] \setminus [\max\{i, j\}]} = e_{[n] \setminus [i]} \} = j.
\]
Hence, Proposition~\ref{L_leftmul} does not hold for $i = n$.
In contrast to the case $i < n$, left multiplication by $a_n$ can strictly decrease the value of $\mathcal{L}$.

\begin{definition}\label{g_definition}
    We define the function $g : \{0, 1,\dots, n\} \times [n] \to \{0, 1, \dots, n\}$ by
    \[
    g(i, j) =
    \begin{cases}
        i - 1, & i = j, \\
        i, & i \neq j.  
    \end{cases}
    \]
\end{definition}

We can rewrite Theorem~\ref{RL_formula} in terms of $g$.
Let $x \in K_n$ and $i \in [n]$.
Then Theorem~\ref{RL_formula} states that
\begin{equation}\label{L-g_formula}
    \mathcal{L}(x a_i) = g(\mathcal{L}(x), i).
\end{equation}

Let $x \in K_n$.
If $x = e$, then $\mathcal{L}(x) = n$ by Lemma~\ref{L_equal_n}.
Otherwise, write $x = a_{i_1} a_{i_2} \dots a_{i_k}$.
Set $l_0 = \mathcal{L}(e) = n$.
By~\eqref{L-g_formula}, we have 
\[
\mathcal{L}(a_{i_1}) = \mathcal{L}(e a_{i_1}) = g(\mathcal{L}(e), i_1) = g(l_0, i_1).
\]
We denote $l_1 = \mathcal{L}(a_{i_1})$ and then we further have
\[
\mathcal{L}(a_{i_1} a_{i_2}) = g(\mathcal{L}(a_{i_1}), i_2) = g(l_1, i_2).
\]
Now let $l_2 = \mathcal{L}(a_{i_1} a_{i_2})$ and we have 
\[
\mathcal{L}(a_{i_1} a_{i_2} a_{i_3}) = g(\mathcal{L}(a_{i_1} a_{i_2}), i_3) = g(l_2, i_3),
\]
and so on.

We have constructed a finite sequence $l_0, l_1, \dots, l_k$ defined recursively by
\[
l_j = 
\begin{cases}
    n, & j = 0, \\
    g(l_{j-1}, i_j), & j = 1, 2, \dots, k.
\end{cases} 
\]
Formula~\eqref{L-g_formula} implies that
\[
\mathcal{L}(x) = l_k.
\]
Hence this provides an algorithm to compute $\mathcal{L}(x)$.

For $x \in K_n$, we denote 
\begin{equation}\label{a_b}
    \mathcal{A}_x = \{i \in \{0, 1, \dots, n\} \mid \overline{\partial}_{[i]}(x) = e_{[n] \setminus [i]} \}, \quad \mathcal{B}_x = \{i \in \{0, 1, \dots, n\} \mid x e_{[i]} = f\}.
\end{equation}

\begin{theorem}\label{a_b_theorem}
    Let $x \in K_n$. Then $\mathcal{A}_x = \mathcal{B}_x$.
\end{theorem}

\begin{proof}
    If $x e_{[i]} = f$, then by applying $\overline{\partial}_{[i]}$ and using $\overline{\partial}_{[i]}(e_{[i]}) = e$, we obtain
    \[
    \overline{\partial}_{[i]}(x) = \overline{\partial}_{[i]}(x) \overline{\partial}_{[i]}(e_{[i]}) = \overline{\partial}_{[i]}(x e_{[i]}) = \overline{\partial}_{[i]}(f) = e_{[n] \setminus [i]},
    \]
    and hence $\mathcal{B}_x \subseteq \mathcal{A}_x$.

    If $\overline{\partial}_{[i]}(x) = e_{[n] \setminus [i]}$, then Lemma~\ref{lemma_chain:2} with $j = i$ implies 
    \[
    x e_{[i]} = \overline{\partial}_{[i]}(x) e_{[i]} = e_{[n] \setminus [i]} e_{[i]} = e_{[n]} = f,
    \] 
    which proves $\mathcal{A}_x \subseteq \mathcal{B}_x$.
    Hence $\mathcal{A}_x = \mathcal{B}_x$.
\end{proof}

Recall the function $m : K_n \to \{0, 1, \dots, n\}$ defined in~\eqref{m_definition}.
\begin{theorem}\label{ml_theorem}
    We have that $\mathcal{L} = m$.
\end{theorem}

\begin{proof}
    Let $x \in K_n$.
    By the definition of $\mathcal{L}(x)$, we have $\mathcal{L}(x) = \min \mathcal{A}_x$.
    By the definition of $m(x)$, we have $m(x) = \min \mathcal{B}_x$.
    By Theorem~\ref{a_b_theorem}, we have $\mathcal{A}_x = \mathcal{B}_x$, and hence $\mathcal{L}(x) = m(x)$.
\end{proof}

%----------------------------
\section{Sequences of Products}\label{INF}

%-----------------------------------
\subsection{Sequences of Partial Products}\label{FINITE}
In this subsection, we study sequences of partial products in $K_n$.
We prove that every such sequence is eventually constant, and in certain cases we determine its eventual value.

Let $(x_j)_{j \geq 1}$ be a sequence with values in $\{a_1, a_2, \dots, a_n\}$.
We define the sequence of partial products $(s_j)_{j \geq 0}$ by
\[
s_j = x_1 x_2 \dots x_j
\]
for $j \geq 1$ and set $s_0 = e$.
Then $s_1 = x_1$ and $s_j = s_{j-1} x_j$ for $j \geq 1$.

\begin{proposition}\label{FS}
    For every sequence $(x_j)_{j \geq 1}$ with values in $\{a_1, a_2, \dots, a_n\}$, there exists $m_0 \in \mathbb{N}$ such that
    \[
    s_m = s_{m_0},
    \]
    for all $m \geq m_0$.
\end{proposition}

\begin{proof}
    Suppose $s_j \neq s_{j-1}$ for some $j \geq 1$.
    Since $s_j = s_{j-1} x_j$, we have
    \[
    s_{j-1} x_j \neq s_{j-1}.
    \]
    Applying the antiautomorphism $\tau$, we obtain
    \[
    \tau(x_j) \tau(s_{j-1}) \neq \tau(s_{j-1}).
    \]
    Applying Lemma~\ref{lemma:height} with $x = \tau(x_j)$ and $y = \tau(s_{j-1})$, we get
    \[
    h(\tau(s_j)) = h(\tau(x_j) \tau(s_{j-1})) < h(\tau(s_{j-1})).
    \]

    Since $h$ takes values in $\mathbb{N}_0$, there can only be finitely many indices $j \geq 1$ such that $s_j \neq s_{j-1}$.
    This concludes the proof.
\end{proof}

For a sequence $(x_j)_{j \geq 1}$, we define
\[
M_{(x_j)} = \{i \in [n] \mid x_k = a_i \text{ for some } k \geq 1\},
\]
and
\[
N_{(x_j)} = \{i \in [n] \mid x_k = a_i \text{ for infinitely many } k \geq 1\}.
\]
Clearly, $N_{(x_j)} \subseteq M_{(x_j)}$.

\begin{theorem}\label{TFS}
    Let $(x_j)_{j \geq 1}$ be a sequence with values in $\{a_1, a_2, \dots, a_n\}$, and assume that $N_{(x_j)} = M_{(x_j)}$.
    Then there exists $m_0 \in \mathbb{N}$ such that
    \[
    s_m = e_{M_{(x_j)}},
    \]
    for all $m \geq m_0$.
\end{theorem}

\begin{proof}
    By Proposition~\ref{FS}, there exists $m_0 \in \mathbb{N}$ such that $s_m = s_{m_0}$ for all $m \geq m_0$.
    Fix $k \in M_{(x_j)}$.
    Since $M_{(x_j)} = N_{(x_j)}$, there exists $i > m_0$ such that $x_i = a_k$.
    Since $i-1 \geq m_0$, we have
    \[
    s_{m_0} = s_{i} = s_{i-1} x_i = s_{m_0} a_k.
    \]
    Since $k \in M_{(x_j)}$ is arbitrary, we have 
    \[
    s_{m_0} a_k = s_{m_0},
    \]
    for all $k \in M_{(x_j)}$.
    Hence, for all $i_1, i_2, \dots, i_l \in M_{(x_j)}$, we have 
    \begin{align}
        s_{m_0} &= s_{m_0} a_{i_l} \nonumber \\
                &= s_{m_0} a_{i_{l-1}} a_{i_l} \nonumber \\
                &= \cdots \nonumber \\
                &= s_{m_0} a_{i_1} a_{i_2} \dots a_{i_l}. \label{z0}
    \end{align}
    Using the content map yields
    \[
    c(s_{m_0}) = c(s_{m_0}) \cup \{i_1, i_2, \dots, i_l\},
    \]
    and hence $M_{(x_j)} \subseteq c(s_{m_0})$.
    By the definition of $M_{(x_j)}$, we have $c(s_{m_0}) \subseteq M_{(x_j)}$, and hence $c(s_{m_0}) = M_{(x_j)}$.
    Hence, using~\eqref{z0}, we obtain 
    \[
    s_{m_0} = s_{m_0}^{k'},
    \]
    for all $k' \geq 1$.
    Since $c(s_{m_0}) = M_{(x_j)}$ and $M_{(x_j)} \neq \emptyset$, by Lemma~\ref{idempotent:power}, we further have 
    \[
    s_{m_0} = s_{m_0}^{|c(s_{m_0})|} = e_{M_{(x_j)}},
    \]
    which concludes the proof.
\end{proof}

\begin{corollary}
    Let $(x_j)_{j \geq 1}$ be a sequence with values in $\{a_1, a_2, \dots, a_n\}$ and assume that $N_{(x_j)} = [n]$.
    Then there exists $m_0 \in \mathbb{N}$ such that
    \[
    s_m = f,
    \]
    for all $m \geq m_0$.
\end{corollary}

\begin{proof}
    If $N_{(x_j)} = [n]$, then $N_{(x_j)} = M_{(x_j)} = [n]$.
    The result then follows from Theorem~\ref{TFS}.
\end{proof}

\subsection{Sequences of Random Partial Products}\label{RANDOM}
In Subsection~\ref{FINITE}, we showed that every sequence of partial products in $K_n$ is eventually constant.
In this subsection, we study such sequences in a probabilistic setting
and investigate the rate at which they attain their final value.

Let $(\Omega, \mathcal{F}, \mathbb{P})$ be a probability space, and let $(X_j)_{j \geq 1}$ be
a sequence of random variables taking values in $\{a_1, a_2, \dots, a_n\}$.
For each $j \geq 1$, let $I_j$ be the random variable with values in $[n]$ such that
\[
X_j = a_{I_j}.
\]

Let $(P_j)_{j \geq 0}$ be random variables taking values in $K_n$, defined by $P_0 = e$ and
\[
P_j = X_1 X_2 \dots X_j = P_{j-1} X_j,
\]
for $j \geq 1$, where the product is taken in $K_n$.
We define the hitting time of state $f$ of the random process $(P_j)_{j \geq 0}$ as 
\[
T = \inf \{j \geq 0 \mid P_j = f\}.
\]
The goal of this subsection is to understand the random variable $T$.
Also observe that if $P_{j_0} = f$ for some $j_0 \geq 1$, then $P_j = f$ for all $j \geq j_0$.
In other words, if the process $(P_j)_{j \geq 0}$ reaches state $f$, then it remains there indefinitely.

We define random variables $(L_j)_{j \geq 0}$ with values in $\{0, 1, \dots, n\}$ by 
\[
L_j = \mathcal{L}(P_j),
\]
for all $j \geq 0$.
Then~\eqref{L-g_formula} implies that, for each $j \geq 1$,
\begin{equation}\label{L_recursion}
    L_j = \mathcal{L}(P_j) = \mathcal{L}(P_{j-1}X_j) = \mathcal{L}(P_{j-1}a_{I_j}) = g(\mathcal{L}(P_{j-1}), I_j) = g(L_{j-1}, I_j).
\end{equation}

For each $i \in [n]$, set
\[
p_i = \mathbb{P}(X_1 = a_i)
\qquad \text{and} \qquad
q_i = \mathbb{P}(X_1 \neq a_i) = 1 - p_i.
\]

\begin{theorem}\label{MC:2}
    Assume that $(X_j)_{j \geq 1}$ are independent and identically distributed.
    Then $(L_j)_{j \geq 0}$ is a Markov chain with state space $\{0, 1, \dots, n\}$.
    The initial distribution vector $\pi$ is
    \[
    \pi = (0, 0, \dots, 0, 1),
    \]
    and the transition probability matrix $\mathcal{P}$ is given by 
    \begin{equation}\label{transition:matrix}
    \mathcal{P} =
    \begin{pmatrix}
    1      & 0      & 0      & \cdots & \cdots & 0       & 0      \\
    p_1    & q_1    & 0      & \cdots & \cdots & 0       & 0      \\
    0      & p_2    & q_2    & \cdots & \cdots & 0       & 0      \\
    \vdots & \vdots & \ddots & \ddots & \vdots & \vdots  & \vdots \\
    \vdots & \vdots & \vdots & \ddots & \ddots & \vdots  & \vdots \\
    0      & 0      & 0      & \cdots & p_{n-1} & q_{n-1} & 0      \\
    0      & 0      & 0      & \cdots & 0        & p_n     & q_n
    \end{pmatrix},
    \end{equation}
    where the rows and columns are indexed by $0, 1, \dots, n$.
\end{theorem}

\begin{proof}
    Since $(X_j)_{j \geq 1}$ are independent and identically distributed, so are $(I_j)_{j \geq 1}$.
    Since for each $j \geq 1$, $I_j$ is independent of $X_1, X_2, \dots, X_{j-1}$ and $L_0, L_1, \dots, L_{j-1}$ are functions of $X_1, X_2, \dots, X_{j-1}$, it follows that $I_j$ is independent of $L_0, L_1, \dots, L_{j-1}$.
    Hence,~\eqref{L_recursion} implies that $(L_j)_{j \geq 0}$ is a Markov chain.

    Since $L_0 = \mathcal{L}(P_0) = \mathcal{L}(e) = n$ is constant, we have $\mathbb{P}(L_0 = n) = 1$, and hence the initial distribution vector $\pi$ is 
    \[
    \pi = (0, 0, \dots, 0, 1).
    \]

    From~\eqref{L_recursion} and the fact that $(I_j)_{j \geq 1}$ are identically distributed, it also follows that the transition probability from state $i$ to state $j$ is given by 
    \begin{equation}\label{tr:prob}
        \mathcal{P}_{i, j} = \mathbb{P}(g(i, I_1) = j).
    \end{equation}
    If $i = 0$, then $g(i, I_1) = 0$, since $I_1$ takes values in $[n]$, and hence $\mathcal{P}_{0, 0} = 1$.
    Now assume that $i > 0$.
    By the definition of $g$, we can write
    \[
    g(i, I_1) = i - \mathbbm{1}_{\{I_1 = i\}}.
    \]
    Hence
    \begin{align*}
        \mathcal{P}_{i, j} &= \mathbb{P}(g(i, I_1) = j) \\
            &= \mathbb{P}(i - \mathbbm{1}_{\{I_1 = i\}} = j) \\
            &= \mathbb{P}(\mathbbm{1}_{\{I_1 = i\}} = i - j).
    \end{align*}
    Since $\{I_1 = i\} = \{X_1 = a_i\}$, we get
    \[
    \mathcal{P}_{i, j} = \mathbb{P}(\mathbbm{1}_{\{I_1 = i\}} = i - j) = \mathbb{P}(\mathbbm{1}_{\{X_1 = a_i\}} = i - j).
    \]
    By noting that $\mathbbm{1}_{\{X_1 = a_i\}}$ takes values in $\{0, 1\}$ with probability $1$, we further have
    \[
    \mathcal{P}_{i, j}
    = 
    \begin{cases}
        0, & i - j \notin \{0, 1\}, \\
        \mathbb{P}(\mathbbm{1}_{\{X_1 = a_i\}} = 1), & i-j = 1, \\
        \mathbb{P}(\mathbbm{1}_{\{X_1 = a_i\}} = 0), & i-j = 0,
    \end{cases}
    = 
    \begin{cases}
        0, & i - j \notin \{0, 1\}, \\
        \mathbb{P}(X_1 = a_i), & i-j = 1, \\
        \mathbb{P}(X_1 \neq a_i), & i-j = 0.
    \end{cases}
    \]
    Therefore,
    \[
    \mathcal{P}_{i, j} = 
    \begin{cases}
        0, & j > i \text{ or } j < i - 1, \\
        p_i, & j = i - 1, \\
        q_i, & j = i.
    \end{cases}
    \]
    This proves~\eqref{transition:matrix} and concludes the proof.
\end{proof}

For $i \in \{0, 1, \dots, n\}$, define 
\[
T_i = \inf \{j \geq 0 \mid L_j = i\},
\]
and also set
\[
V_i = T_{i-1} - T_i,
\]
for $i = 1, 2, \dots, n$.
In the following proof, we show that all $V_i$ are well-defined under the additional assumption 
$p_j > 0$ for all $j = 1, 2, \dots, n$.
 
\begin{proposition}\label{prop:geom}
    Assume that $(X_j)_{j \geq 1}$ are independent and identically distributed and assume $p_i > 0$ for all $i = 1, 2, \dots, n$.
    Then $V_n, V_{n-1}, \dots, V_1$ are independent and $V_i$ is a geometric random variable with success probability $p_i$.
\end{proposition}

\begin{proof}
    By Theorem~\ref{MC:2}, $(L_j)_{j \geq 0}$ is a Markov chain with transition probability matrix $\mathcal{P}$.
    By~\eqref{L_recursion}, we have
    \[
    L_j \leq L_{j-1} \quad \text{and} \quad L_{j-1} - L_j \in \{0, 1\},
    \]
    for all $j \geq 1$.
    Hence at each step, the chain either remains in the previous state or decreases its state by exactly one.
    Since $p_i = \mathcal{P}_{i, i-1} > 0$, we have $\mathcal{P}_{i, i} = q_i < 1$, and thus $\lim_{k \to \infty} \mathcal{P}_{i, i}^k = 0$.
    Hence the probability that the chain stays at state $i \in [n]$ indefinitely is zero.
    Hence, every $T_i$ is almost surely finite.
    For $i \in [n]$ and $k \geq 1$, we further have
    \[
    \{V_i = k\} = \{T_{i-1} -T_i = k\}
    \]
    is the event that once $(L_j)_{j \geq 0}$ reaches state $i$, it remains there for exactly $k-1$ steps, and then jumps to state $i-1$.
    Therefore, by the strong Markov property, we have
    \begin{equation}\label{V:geom}
        \mathbb{P}(V_i = k) = (\mathcal{P}_{i, i})^{k-1} \mathcal{P}_{i, i-1} = q_i^{k-1} p_i.
    \end{equation}
    Hence $V_i$ is a geometric random variable with success probability $p_i$.

    Now we show independence.
    Let $k_n, k_{n-1}, \dots, k_1 \geq 1$.
    Then the event 
    \[
    \{V_n = k_n, V_{n-1} = k_{n-1}, \dots, V_1 = k_1\}
    \]
    is the event that $(L_j)_{j \geq 0}$ starts in state $n$, remains there for $k_n-1$ steps, and then jumps down to state $n-1$.
    Then it stays at state $n-1$ for $k_{n-1} - 1$ steps, and then it jumps down to state $n-2$, and so on.
    When it reaches state $1$, it remains there for $k_1 - 1$ steps and then jumps to $0$.
    Hence, $\mathbb{P}(V_n = k_n, V_{n-1} = k_{n-1}, \dots, V_1 = k_1)$ is equal to 
    \[
    (\mathcal{P}_{n, n})^{k_n -1} \mathcal{P}_{n, n-1} (\mathcal{P}_{n-1, n-1})^{k_{n-1} -1} \mathcal{P}_{n-1, n-2} \dots (\mathcal{P}_{1, 1})^{k_1 -1} \mathcal{P}_{1, 0},
    \]
    which is by~\eqref{V:geom} equal to
    \[
    \mathbb{P}(V_n = k_n) \mathbb{P}(V_{n-1} = k_{n-1}) \dots \mathbb{P}(V_1 = k_1).
    \]
    Hence, $V_n, V_{n-1}, \dots, V_1$ are independent as well.
\end{proof}

\begin{theorem}\label{L_hitting_time}
    Assume that $(X_j)_{j \geq 1}$ are independent and identically distributed and assume $p_i > 0$ for all $i = 1, 2, \dots, n$.
    Then $T$ is distributed as a sum of $n$ independent geometric random variables with success probabilities $p_1, p_2, \dots, p_n$.
\end{theorem}

\begin{proof}
    Since $L_j = \mathcal{L}(P_j)$ for $j \geq 0$, Lemma~\ref{RU1} implies 
    \[
    P_j = f \iff L_j = 0,
    \]
    which implies
    \[
    T = \inf \{j \geq 0 \mid P_j = f\} = \inf \{j \geq 0 \mid L_j = 0\} = T_0.
    \]
    Since $L_0 = n$, we have $T_n = 0$, and hence
    \[
    T = T_0 = \sum_{i=1}^n (T_{i-1} -T_i) = \sum_{i=1}^n V_i.
    \]
    By Proposition~\ref{prop:geom}, the random variables $V_n, V_{n-1}, \dots, V_1$ are independent and $V_i$ is a geometric random variable with success probability $p_i$.
    This finishes the proof.
\end{proof}

\begin{corollary}
    Assume that $(X_j)_{j \geq 1}$ are independent and identically distributed and assume $p_i > 0$ for all $i = 1, 2, \dots, n$.
    Then
    \[
    \mathbb{E}[T] = \sum_{i=1}^n p_i^{-1}.
    \]
    In particular, if $(X_j)_{j \geq 1}$ are uniformly distributed on $\{a_1, a_2, \dots, a_n\}$, then
    \[
    \mathbb{E}[T] = n^2.
    \]
\end{corollary}

\begin{proof}
    This result follows from Theorem~\ref{L_hitting_time} and the well-known fact that if $X$ is a geometric
    random variable with success probability $p$, then $\mathbb{E}[X] = p^{-1}$.
\end{proof}

%~\ref

%|-A~\
%----------------------------------
\section{$K_n$ as a Metric Space}\label{MET}

In Section~\ref{L_Function}, we defined the level function $\mathcal{L}$.
We can rewrite its definition as
\[
\mathcal{L}(x) = \min\{i \in \{0, 1, \dots, n\} \mid \overline{\partial}_{[i]}(x) = \overline{\partial}_{[i]}(f)\}.
\]

This gives rise to the definition of the map $d : K_n \times K_n \to \{0, 1, \dots, n\}$:
\begin{equation}\label{Kn_Metric}
    d(x, y) = \min\{i \in \{0, 1, \dots, n\} \mid \overline{\partial}_{[i]}(x) = \overline{\partial}_{[i]}(y) \}.
\end{equation}
This map is well-defined, since $\overline{\partial}_{[n]}(x) = \overline{\partial}_{[n]}(y)$ for any $x, y \in K_n$.
We also have $\mathcal{L}(x) = d(x, f)$.

\begin{theorem}
    The pair $(K_n, d)$ is an ultrametric space.
\end{theorem}

\begin{proof}
    Let $x, y \in K_n$.
    We of course have $d(x, y) \geq 0$.
    We have $d(x, y) = 0$ if and only if $\overline{\partial}_{[0]}(x) = \overline{\partial}_{[0]}(y)$.
    Since $x = \overline{\partial}_{[0]}(x)$ and $y = \overline{\partial}_{[0]}(y)$, the condition $d(x, y) = 0$ is equivalent to $x=y$.

    By the definition of $d$, we have $d(x, y) = d(y, x)$.

    Let $z \in K_n$.
    We set $i = d(x, z)$, $j = d(z, y)$, and $k = \max\{i, j\}$. 
    Since $[i], [j] \subseteq [k]$, we have
    \begin{align*}
        \overline{\partial}_{[k]}(x) 
        &= \overline{\partial}_{[k]}(\overline{\partial}_{[i]}(x)) \\
        &= \overline{\partial}_{[k]}(\overline{\partial}_{[i]}(z)) \\
        &= \overline{\partial}_{[k]}(z) \\
        &= \overline{\partial}_{[k]}(\overline{\partial}_{[j]}(z)) \\
        &= \overline{\partial}_{[k]}(\overline{\partial}_{[j]}(y)) \\
        &= \overline{\partial}_{[k]}(y).
    \end{align*}
    Hence we obtain
    \[
    d(x, y) \leq k = \max\{d(x, z), d(z, y)\},
    \]
    which concludes the proof.
\end{proof}

For $a \in K_n$ and $r \in \{0, 1, \dots, n\}$, we denote by
\begin{align*}
    B(a, r) &= \{x \in K_n \mid d(a, x) \leq r\}, \\
    S(a, r) &= \{x \in K_n \mid d(a, x) = r\},
\end{align*}
the metric ball and metric sphere with centre $a$ and radius $r$, respectively.

\begin{lemma}\label{ball_lemma}
    We have that $x \in B(f, r)$ if and only if 
    $x e_{[r]} = f$.
\end{lemma}

\begin{proof}
    If $x \in B(f, r)$, then $\mathcal{L}(x) \leq r$ and by Corollary~\ref{RL2}, we have 
    \[
    \overline{\partial}_{[r]}(x) = e_{[n] \setminus [r]}.
    \]
    By Theorem~\ref{a_b_theorem}, we then have $x e_{[r]} = f$.
    
    If $x e_{[r]} = f$, then $m(x) \leq r$, and hence by Theorem~\ref{ml_theorem}, we have $\mathcal{L}(x) \leq r$, and hence $x \in B(f, r)$.
\end{proof}

\begin{corollary}
    We have that $|B(f, 1)| = 1 + |K_{n-1}|$ and consequently $|S(f, 1)| = |K_{n-1}|$.
\end{corollary}

\begin{proof}
    By Lemma~\ref{ball_lemma}, we have 
    \begin{equation*}
        B(f, 1) = \{x \in K_n \mid x a_1 = f\}.
    \end{equation*}
    By Proposition~\ref{bc1}, it follows that $|B(f, 1)| = 1 + |K_{n-1}|$.
    Since $S(f, 1) = B(f, 1) \setminus \{f\}$, we have $|S(f, 1)| = |K_{n-1}|$.
\end{proof}

\begin{proposition}\label{easy_prop}
    We have that $S(f, n) = \{x \in K_n \mid c(x) \subseteq [n-1]\}$.
\end{proposition}

\begin{proof}
    This follows immediately from Lemma~\ref{L_equal_n}.
\end{proof}

\begin{corollary}
    We have $|S(f, n)| = |K_{n-1}|$.
\end{corollary}

\begin{proof}
    The proof is immediate from Proposition~\ref{easy_prop} and the
    fact that 
    $\{x \in K_n \mid c(x) \subseteq [n-1]\}$ is in bijection with $K_{n-1}$.
\end{proof}

\bmhead{Acknowledgements}

I am thankful to my close family and friends for their support, help, and motivation.
Thank you for everything.

\bibliography{sn-bibliography.bib}

\end{document}